\documentclass[a4paper]{article}

\usepackage{amsmath}
\usepackage{amssymb}
\usepackage{amsthm}
\usepackage{mathrsfs}

\theoremstyle{plain}
\newtheorem{thm}{Theorem}[section]
\newtheorem{prop}[thm]{Proposition}
\newtheorem{lem}[thm]{Lemma}
\newtheorem{cor}[thm]{Corollary}

\theoremstyle{definition}

\theoremstyle{remark}
\newtheorem{rem}[thm]{Remark}

\numberwithin{equation}{section}

\DeclareMathOperator{\dom}{dom}		
\DeclareMathOperator{\tr}{tr}		

\def \bR {\mathbb R}		

\def \bP {\mathbb P}		%

\def \cE {\mathcal E}		%
\def \cH {\mathcal H}		%
\def \cK {\mathcal K}		%
\def \cL {\mathcal L}		%
\def \cP {\mathcal P}		%
\def \cR {\mathcal R}		%
\def \sF {\mathscr F}		%
\begin{document}

\pagestyle{plain}
\bibliographystyle{plain}

\title{An invariance principle for stochastic heat equations \\with periodic coefficients}
\author{Lu \textsc{Xu}\footnote{l-xu@math.kyushu-u.ac.jp} \quad \textit{Kyushu University}}
\date{}
\maketitle

\noindent{\bf Abstract} We investigate the asymptotic behaviors of the solution $u(t, \cdot)$ to a stochastic heat equation with a periodic, gradient-type nonlinear term. We extend the central limit theorem for finite-dimensional diffusions presented in \cite[\S9.1]{KLO12} to infinite-dimensional settings. Due to our results, as $t \rightarrow \infty$, $\frac{1}{\sqrt{t}}u(t, \cdot)$ converges weakly to a centered Gaussian variable whose covariance operator is described through Poisson's equations. Different from the finite-dimensional case, the fluctuation in space vanishes in the limit distribution. Furthermore, we verify the tightness and present an invariance principle for $\{\epsilon u(\epsilon^{-2}t, \cdot)\}_{t \in [0, T]}$ as $\epsilon \downarrow 0$. 





\begin{section}{Introduction and main results}

A central limit theorem for additive functionals of reversible Markov chains and processes is established in \cite{KV86}, based on a martingale-decomposition of the targeted functional. Their method is extended to non-reversible cases with various kinds of conditions (e.g., \cite{LY97,MW00,OS95,V95}). Combined with It\^{o} formula, this method can be used to prove the central limit theorem for finite-dimensional diffusion processes with periodic coefficients, as illustrated in \cite[\S9.1]{KLO12}. In the present paper, we extend this strategy to the case of infinite-dimensional diffusion processes. 

Consider here a stochastic PDE on the unit interval $[0, 1]$ with a gradient-type nonlinear term. Precisely, suppose $(\Omega, \sF, \bP)$ to be a complete probability space equipped with a filtration of $\sigma$-fields $\{\sF_t \subseteq \sF; t \geq 0\}$ satisfying the usual conditions, and $W_t = W(t, \cdot)$ to be a standard cylindrical Brownian motion on $(\Omega, \sF, \bP)$ adapted to $\{\sF_t\}$. We consider the equation 
\begin{equation} \label{spde} \left\{
\begin{aligned}
&\partial_t u(t, x) = \frac{1}{2} \partial_x^2 u(t, x) - V'_x(u(t, x)) + \dot W(t, x), &&t > 0, x \in (0, 1), \\
&\partial_x u(t, 0) = \partial_x u(t, 1) = 0, &&t > 0, \\
&u(0, x) = v(x), &&x \in [0, 1], 
\end{aligned} \right. 
\end{equation}
where $V_x$ is a $C^1$ function on $\bR$ for each $x \in [0, 1]$ and $V'_x(u) \triangleq \frac{d}{du}V_x(u)$. Equation \eqref{spde} has been studied in \cite{F83} for the purpose of describing the motion of a flexible Brownian string in a potential field. We assume the following conditions on $V$: 
\begin{itemize}
\item [(1)] $\forall u \in \bR$, $V_x(u)$ is Borel-measurable in $x$; 
\item [(2)] $\sup_{x \in [0, 1], u \in \bR}\left\{|V_x(u)| + |V'_x(u)|\right\} < \infty$; 
\item [(3)] $\exists$ a constant $C$, s.t. $|V'_x(u_1) - V'_x(u_2)| \leq C|u_1 - u_2|$, $\forall x \in [0, 1]$, $\forall u_1$, $u_2 \in \bR$. 
\item [(4)] $\forall x \in [0, 1]$, $V_x$ is periodic in $u$: $V_x(u) = V_x(u + 1)$. 
\end{itemize}

The motivation of studying \eqref{spde} lies in several aspects. First of all, the study on finite-dimensional diffusions shows the links between periodic and random coefficients. Thus, our results can provide some hints for the more complex case with random coefficients, via the technique of the environment seen from the observer. Second, by taking $V_x$ to be cosine functions, \eqref{spde} appears closely related to the dynamical sine-Gordon equation 
$$\partial_tu = \frac{1}{2}\Delta u + c\sin(\beta u + \theta) + \xi, $$
where $c$, $\beta$ and $\theta$ are real constants and $\xi$ is the standard space-time white noise. It is the natural dynamic associated with the usual quantum sine-Gordon model, see \cite{HS16}. From a more physical perspective, such dynamic in space dimension two is first discussed in \cite{CW78} to study the crystal growth near the roughening transition. It describes globally neutral gas of interacting charges at different temperature $\beta$, see \cite{HS16}. Here to avoid the ill-posedness of \eqref{spde} in multi-dimensional space, we restrict our discussion to one-dimensional case. 

The solution to \eqref{spde} studied in this paper is the mild solution defined as follows. Let $S(t)$ be the semigroup on $L^2([0, 1])$ generated by $\frac{1}{2}\partial_x^2$ with the Neumann boundary conditions. Then, a process $u(t) = u(t, \cdot)$ is called a mild solution to \eqref{spde} if for all $t \geq 0$, 
\begin{equation} \label{mild solution}
u(t) = S(t)v + \int_0^t S(t-r)[-V'_\cdot(u(r, \cdot))]dr + \int_0^t S(t-r)dW_r. 
\end{equation}
The existence and regularities of $u(t)$ have been discussed in \cite{F83,F91}. With conditions (1), (2), and (3), due to \cite[Theorem 2.1]{F83}, $u(t)$ uniquely exists in $C[0, 1]$ and forms a continuous Markov process. Furthermore, by \cite[Theorem 4.1]{F83}, $u(t)$ possesses an (infinite) reversible measure given by 
\begin{equation} \label{reversible measure}
\mu(dv) = \exp \left\{-2\int_0^1 V_x(v(x)) dx\right\} \mu_w(dv), 
\end{equation}
where $\mu_w$ stands for the measure on $C[0, 1]$ induced by one-dimensional Brownian motion $\{w_x; x \in [0, 1]\}$ whose initial distribution is the Lebesgue measure on $\bR$. 

Our main results are an annealed central limit theorem and an invariance principle for $u(t, \cdot)$ stated in Theorem \ref{clt} and \ref{invariance principle}. 

\begin{thm} \label{clt}
Suppose that the initial condition of $u(t, \cdot)$ is a probability measure $\nu$ absolutely continuous with respect to $\mu$ defined in \eqref{reversible measure}. Then, for all $f \in C_b(C[0, 1])$, 
\begin{equation} \label{clt-equation}
\lim_{t \rightarrow \infty}E_\nu \left|E_\bP \left[\left.f\left(\frac{u(t)}{\sqrt{t}}\right)\right|\sF_0\right] - \int_{\bR} f(y\mathbf{1})\Phi_\sigma(dy)\right| = 0, 
\end{equation}
where $y\mathbf{1}$ is the function on $[0, 1]$ taking constant value $y$, and $\Phi_\sigma$ is the density function of a one-dimensional centered Gaussian law with variance $\sigma^2$ defined in \eqref{diffusion constant}. 
\end{thm}

\begin{rem}
In Theorem \ref{clt} the fluctuation in $x$ decays asymptotically and only the height of solution remains in the limit distribution. This kind of decay occurs because the Laplacian controls the growth of the fluctuation. More precisely, if one adopts Dirichlet boundary conditions in \eqref{spde} instead of Neumann conditions, the periodicity of nonlinear term turns to be unnecessary, and $u(t, \cdot)$ converges weakly to a Brownian bridge. 
\end{rem}

\begin{thm} \label{invariance principle}
Fix some $T > 0$ arbitrarily. Under initial distribution $\nu$ which is absolutely continuous with respect to $\mu$, $\{\epsilon u(\epsilon^{-2}t), t \in [0, T]\}$ converges weakly to a Gaussian process $\{\sigma B_t \cdot \mathbf{1}, t \in [0, T]\}$ as $\epsilon \downarrow 0$, where $B_t$ is a standard one-dimensional Brownian motion on $[0, T]$ and $\sigma$ is the constant defined in \eqref{diffusion constant}. 
\end{thm}

The paper is arranged as follows. In \S2, we study the details of the Kolmogrov generator related to \eqref{spde} and present an It\^{o} formula necessary for our proof in Corollary \ref{Ito}. In \S3, we state the proofs of the main results along the ideas in \cite{KV86} and \cite{KLO12}. 

We end this section by introducing some useful notations. Throughout this paper, $E$ is the Banach space $C[0, 1]$, and $E_0$ is its subspace consisting of twice continuously differentiable functions $\varphi$ such that $\varphi'(0) = \varphi'(1) = 0$. $H$ is the Hilbert space $L^2[0, 1]$, equipped with the inner product $\langle \cdot, \cdot \rangle$ and the corresponding norm $\|\cdot\|$. Let $e_0 = \mathbf{1}$ and $e_j(x) = \sqrt{2}\cos(\pi jx)$ for $j \geq 1$ and $x \in [0, 1]$, then $e_j \in E_0$ and $\{e_j\}_{j=0}^\infty$ forms a complete orthonormal system of $H$. 

\end{section}

\begin{section}{Reduced process and its It\^o formula}

Before discussing the generator of \eqref{spde}, we first realize a reduction of $E$. Consider an equivalence relation in $E$ such that $v_1 \sim v_2$ if and only if $v_1-v_2$ equals to some integer-valued constant function. Define $\dot v$ to be the equivalence class of $v$, and $\dot{E}$ to be the quotient space $E/\sim$. From now on we identify $\dot v$ with its representative $v$ such that $v(0) \in [0, 1)$. A function $f$ on $E$ which satisfies that $f(v + \mathbf{1}) = f(v)$ can be automatically regarded as a function on $\dot{E}$, and vice versa. Furthermore, given a function $f$ on $\dot E$ which can be continuously extended to a Fr\'echet differentiable function $\bar{f}$ on $H$ with derivative $D\bar{f}$, we simply write $Df$ in short of $D\bar{f}|_{\dot E}$, and call it the derivative of $f$ on $E$. 

We explain the benefits of adopting such a reduction. To specify the initial condition in \eqref{spde}, we sometimes write $u^v(t)$ instead of $u(t)$. Recalling that $V_x$ is periodic with period 1, the uniqueness of the mild solution shows that $u^{v+\mathbf{1}}(t) = u^v(t) + \mathbf{1}$ almost surely. This makes the following ``reduced process" be well-defined: 
$$\dot u(t) \triangleq \dot{(u(t))}, \quad \forall t \in [0, +\infty), $$
where the right-hand side means the equivalence class of $u(t)$ for each $\omega \in \Omega$. This reduction then defines an $\dot E$-valued Markov process and allows us to convert the infinite measure $\mu$ defined in \eqref{reversible measure} into a finite one. Precisely, suppose $\{w'_x\}_{x \in [0, 1]}$ to be a one-dimensional Brownian motion whose initial distribution is the Lebesgue measure on $[0, 1)$, then 
$$\pi(d\dot{v}) = \frac{1}{Z} \exp \left\{-2\int_0^1 V_x(\dot{v}(x)) dx \right\}\pi_w(d\dot v), $$
where $\pi_w$ stands for the distribution of $w'_x$ on $\dot E$ and $Z$ is the normalization constant. Apparently, $\dot u(t)$ is reversible under $\pi$. 

For further discussion on the ergodicity and generator, consider an abstract stochastic differential equation in the Hilbert space $H$, written as 
\begin{equation} \label{spde-abstract} \left\{
\begin{aligned}
&dX(t) = [AX(t) - DV(X(t))]dt + dW_t, \\
&X(0) = h \in H, 
\end{aligned} \right. 
\end{equation}
where $A = \frac{1}{2}\partial_x^2$ with domain $\dom(A) = E_0$, $V(h) = \int_0^1 V_x(h(x))dx$ for $h \in H$ and $DV$ stands for the Fr\'echet derivative of $V$. With conditions (2) and (3), the mild solution $X(t)$ to \eqref{spde-abstract} exists uniquely in $H$, see, e.g., \cite[p. 194, Theorem 7.6]{PZ92}. 

The probability measure $\pi$ is ergodic. Indeed, denote the Hilbert-Schmidt norm by $\|\cdot\|_2$ and we know that $\int_0^t \|S(r)\|_2^2dr < \infty$ for all $t > 0$, thus $X(t)$ is strong Feller and irreducible, see \cite{PZ95}. Such process is regular, i.e., all of its transition probabilities are equivalent to each other. Since $u(t)$ in \eqref{spde} coincides with $X(t)$ when $h = v \in E$, $u(t)$ and $\dot u(t)$ are also regular. Consequently, $\pi$ is the unique invariant probability measure of $\dot u(t)$. 

Next we discuss the generator of the semigroup determined by $\dot u(t)$ on $\cH = L^2(\dot E, \pi)$. We write the inner product in $\cH$ by $\langle \cdot, \cdot \rangle_\pi$ and the norm by $\|\cdot\|_\pi$. Let $\dot\cP_t$ be the semigroup induced by $\dot u(t)$ on $\cH$. Our aim is to define a Markov pre-generator on a relatively simple subspace of $\cH$ and prove that its closure generates $\dot\cP_t$. Notice that similar discussion can be found in \cite{PT01}, where they consider the dissipativity of a kind of perturbed Kolmogorov operators with finite invariant measures. We adopt their ideas, however in our model 
$$\sup_{t > 0}\left\{\tr\left[\int_0^t S(r)S^*(r)dr\right]\right\} = \sup_{t > 0}\left\{t + \sum_{j=1}^\infty \frac{1- e^{-\pi^2j^2t}}{\pi^2j^2}\right\} = \infty, $$
which is in contract to \cite[Hypothesis 1(\romannumeral2)]{PT01}. 

Let $\cE_A(H)$ be the linear span of all real and imaginary parts of functions on $H$ of the form $h \mapsto e^{i\langle l, h \rangle}$ where $l \in E_0$, and $\cE_A(\dot E)$ be the collection of functions in $\cE_A(H)$ such that $f(v) = f(v + \mathbf{1})$ for all $v \in E$. $\cE_A(\dot E)$ can be viewed as a sub-algebra of $C_b(\dot E)$. It is not hard to verify that $\cE_A(\dot E)$ is dense in $\cH$ (cf. \cite{MR92}). Consider the following Ornstein-Uhlenbeck semigroup on $\cE_A(H)$: 
$$\cR_tf(h) = \int_H f(S(t)h + l) \mathcal{N}_{Q_t}(dl)$$
where $Q_t = \int_0^t S(r)S^*(r)dr$ and $\mathcal{N}_{Q_t}$ is the centered Gaussian law on $H$ whose covariance operator is $Q_t$. Notice that for any $f_l(h) = e^{i\langle l, h \rangle}$ (\cite[Remark 2.2]{PT01}), 
\begin{equation} \label{computation-core}
\cR_tf_l = e^{-\frac{1}{2}\langle Q_tl, l \rangle}f_{S^*(t)l}, 
\end{equation}
therefore the restriction of $\cR_t$ on $\cE_A(\dot E)$ can then be continuously extended to a contraction $\dot\cR_t$ on $\cH$, and $\{\dot\cR_t\}$ forms a strong continuous semigroup. Denote its infinitesimal generator by $(\dom(\dot\cL), \dot\cL)$. Direct calculation shows that $\cE_A(\dot E) \subset \dom(\dot\cL)$, and 
$$\dot\cL f(\dot v) = \langle A[Df(\dot v)], v \rangle + \frac{1}{2}\tr\left[D^2f(\dot v)\right], \quad \forall f \in \cE_A(\dot E), $$
where $D$ is the Fr\'echet differential operator. Now define an operator $\dot\cK$ on $\cE_A(\dot E)$ by 
\begin{equation} \label{generator-pre}
\dot\cK f(\dot v) =  \dot\cL f(\dot v) - \langle Df(\dot v), DV(\dot v) \rangle, \quad \forall f \in \cE_A(\dot E). 
\end{equation}
The integration-by-part formula for Wiener measure shows that 
\begin{equation} \label{dirichlet form}
E_\pi \|Df\|^2 = 2\langle f, -\dot\cK f \rangle_\pi. 
\end{equation}
From \eqref{dirichlet form}, $\dot\cK $ is dissipative, and thus closable. Denote its closure on $\cH$ still by $\dot\cK $, and its domain by $\dom(\dot\cK)$. In the following part, we prove that $\dot\cK$ generates $\dot\cP_t$. To do this, it is sufficient to prove the maximal dissipativity of $\dot\cK $, i.e., the Poisson's equation 
$$\lambda f - \dot\cK f = g $$
is solvable in $\dom(\dot\cK )$ for all $\lambda > 0$ and $g \in \cH$. 

\begin{prop} \label{maximal dissipativity}
$(\dom(\dot\cK), \dot\cK)$ is maximally dissipative on $\cH$. 
\end{prop}

To prove Proposition \ref{maximal dissipativity}, we make use of the next lemma, originally stated in \cite[Lemma 3.3]{PT01}. In the following contents, $|\cdot|_{L(H,H)}$ stands for the norm on the space of linear functionals from $H$ to $H$, $|D^2V|_\infty = \sup_{h \in H}|D^2V(h)|_{L(H, H)}$ for $V \in C_b^2(H)$ and $\|Df\|_\infty = \sup_{h \in \dot E}\|Df(h)\|$ for $f \in C_b^1(\dot E)$. 

\begin{lem} \label{lemma-core}
If $V \in C_b^2(H)$, $\lambda > |D^2V|_\infty$ and $g \in \cE_A(\dot E)$, then the equation 
\begin{equation} \label{resolvent equation-core}
\lambda f - \dot\cL f + \langle Df, DV \rangle = g 
\end{equation}
has a solution $f \in C_b^1(\dot E) \cap \dom(\dot\cL)$. Furthermore, the following estimate holds: 
\begin{equation} \label{estimate-core}
\|Df\|_\infty \leq \left(\lambda - |D^2V|_\infty\right)^{-1}\|Dg\|_\infty. 
\end{equation}
\end{lem}

\begin{proof}
Fix some $\lambda > |D^2V|_\infty$ and write $E_\bP^{\dot v} = E_\bP[~\cdot~|~\dot u(0) = \dot v]$. Define 
$$f(\dot v) = \int_0^\infty e^{-\lambda r}E_\bP^{\dot v}[g(\dot u(r))]dr. $$
It follows from $V \in C_b^2(H)$ that $f \in C_b^1(\dot E)$ and 
$$\langle Df(\dot v), h \rangle = \int_0^\infty e^{-\lambda r}E_\bP^{\dot v}[\langle Dg(\dot u(r)), Y(r) \rangle]dr, $$
where $Y(r) \in H$ is the solution to the deterministic PDE 
$$dY(r) = AY(r)dr - D^2V(X(r))(Y(r))dt, \quad Y(0) = h. $$
Thus \eqref{estimate-core} follows from the fact that $\|Y(r)\| \leq \exp(r|D^2V|_\infty)\|h\|$. For \eqref{resolvent equation-core}, write 
\begin{equation*}
\begin{aligned}
\dot\cR_tf(\dot v) - f(\dot v) =&\ \left(\dot\cR_tf(\dot v) - E_\bP^{\dot v}[f(\dot u(t))]\right) + \left(E_\bP^{\dot v}[f(\dot u(t))] - f(\dot v)\right) \\
=&\ \mathbb{E}^{\dot v}\left[f\left(u(t) + \int_0^t S(t - r)[DV(u(r))]dr\right) - f(u(t))\right] \\
&+ \int_0^\infty e^{-\lambda r}E_\bP^{\dot v}[g(\dot u(t + r)) - g(\dot u(r))]dr. 
\end{aligned}
\end{equation*}
By elementary calculation, $f \in \dom(\dot\cL)$ and \eqref{resolvent equation-core} holds. 
\end{proof}

\begin{proof}[Proof of Proposition \ref{maximal dissipativity}]
Since we only have $V \in C_b^1(H)$ in the our assumptions, we first choose a sequence $V_m \in C_b^2(H)$ such that $|D^2V_m|_\infty$ are uniformly bounded and 
$$\lim_{m \rightarrow \infty} E_\pi \left[\|DV_m - DV\|^2\right] = 0. $$
Fix $g \in \cE_A(\dot E)$. By Lemma \ref{lemma-core}, we can choose some $\lambda > 0$ such that \eqref{resolvent equation-core} can be solved by some $f_m \in C_b^1(\dot E) \cap \dom(\dot\cL)$ with $V$ replaced by $V_m$. In view of \eqref{computation-core}, $\dot\cR_t[\cE_A(\dot E)] \subseteq \cE_A(\dot E)$, therefore $\cE_A(\dot E)$ forms a core of $\dot\cL$ (\cite[Proposition 1.3.3]{EK05}). Thus, we can pick $f_{m, n} \in \cE_A(\dot E)$ such that both $f_{m, n} \rightarrow f_m$ and $\dot\cL f_{m, n} \rightarrow \dot\cL f_m$ hold in $\cH$ as $n \rightarrow \infty$. Combining these arguments with \eqref{dirichlet form}, we get 
\begin{equation*}
\begin{aligned}
&\limsup_{n_1, n_2 \rightarrow \infty} E_\pi \left[\|Df_{m, n_1} - Df_{m, n_2}\|^2\right] \\
= &\ 2\limsup_{n_1, n_2 \rightarrow \infty} E_\pi \left[\left(f_{m, n_1} - f_{m, n_2}\right)\langle Df_{m, n_1} - Df_{m, n_2}, DV \rangle\right] \\
\leq&\ 2\|DV\|_\infty \liminf_{n_1, n_2 \rightarrow \infty} \left\{\|f_{m, n_1} - f_{m, n_2}\|_\pi \left[E_\pi \|Df_{m, n_1} - Df_{m, n_2}\|^2\right]^{\frac{1}{2}}\right\}, 
\end{aligned}
\end{equation*}
thus $E[\|Df_{m, n_1} - Df_{m, n_2}\|^2]$ vanishes as $n_1$, $n_2 \rightarrow \infty$. Since the Dirichlet form in \eqref{dirichlet form} is closable in $\cH$ (\cite[\uppercase\expandafter{\romannumeral2}.3]{MR92}), we know that $E_\pi [\|Df_{m, n} - Df_m\|^2]$ also vanishes as $n \rightarrow \infty$. Therefore 
$$\lim_{n \rightarrow \infty} \dot\cK f_{m, n} = \dot\cL f_m - \langle Df_m, DV \rangle $$
holds in $\cH$, so that $f_m \in \dom(\dot\cK )$ for all $m$ and $\dot\cK f_m = \dot\cL f_m - \langle Df_m, DV \rangle$. By \eqref{estimate-core}, $\|Df_m\|_\infty$ is uniformly bounded, then we can conclude from the definition of $\dot\cK$ in \eqref{generator-pre} that 
$$\lim_{m \rightarrow \infty} \|g - (\lambda f_m - \dot\cK f_m)\|_\pi^2 \leq \lim_{m \rightarrow \infty} \left\{\|Df_m\|_\infty \cdot E_\pi\|DV_m - DV\|^2\right\} = 0, $$
where the vanishment is due to the assumption on $V_m$. The estimate above shows that the image of $\lambda - \dot\cK$ is dense in $\cH$. Now the maximal dissipativity follows from the classical Lumer-Phillips theorem (see, e.g., \cite[p. 17, Proposition 1.3.1]{EK05} and \cite{Y80}). 
\end{proof}

We have proved that $\{\dot\cP_t\}$ is a strong continuous Markov semigroup on $\cH$, whose infinitesimal generator $(\dom(\dot\cK), \dot\cK)$ has a core $\cE_A(\dot E)$, the linear span of exponential functions. Then from \eqref{dirichlet form}, it is clear that the Fr\'echet differential operator $D$ can be extended to $\dom(\dot\cK)$. We still denote the extended operator by $D$. Based on these results and notations, it is not hard to conclude an It\^o formula for functions $f \in \dom(\dot\cK)$. 

\begin{cor} \label{Ito}
For all $f \in \dom(\dot\cK)$, the following equation 
\begin{equation}\label{ito-equation}
f(\dot u(t)) = f(\dot u(0)) + \int_0^t \dot\cK f(\dot u(r))dr + \int_0^t \langle Df(\dot u(r)), dW_r \rangle 
\end{equation}
holds $\pi$-a.s. and in $\cH$. 
\end{cor}

\begin{proof}
If $f \in \cE_A(\dot E)$, \eqref{ito-equation} indeed only concerns finite dimensions, thus follows from the classical It\^o formula. For general $f \in \dom(\dot\cK )$, pick $f_m \in \cE_A(\dot E)$ such that $f_m$ and $\dot\cK f_m$ converge to $f$ and $\dot\cK f$ in $\cH$, respectively. The corollary then follows from \eqref{dirichlet form} and It\^o isometry. 
\end{proof}

\end{section}

\begin{section}{Proofs of Theorems \ref{clt} and \ref{invariance principle}}

Before illustrating the proof of the central limit theorem, we define two Hilbert spaces related to the operator $\dot\cK $. For $f \in \cE_A(\dot E)$ let 
$$\|f\|_1^2 = \langle -\dot\cK f, f \rangle_\pi = \frac{1}{2}E_\pi\|Df\|^2. $$
Let $\cH_1$ be the completion of $\cE_A(\dot E)$ under $\|\cdot\|_1$, which turns to be a Hilbert space if all functions $f$ such that $\|f\|_1 = 0$ are identified with $0$. On the other hand, let 
$$\mathcal{I} = \left\{f \in \cH; \|f\|_{-1} \triangleq \sup\left\{\langle f, g \rangle_\pi \left| g \in \cE_A(\dot E), \|g\|_1 = 1\right.\right\} < \infty \right\} $$
Let $\cH_{-1}$ be the completion of $\mathcal{I}_{-1}$ under $\|\cdot\|_{-1}$, which also becomes a Hilbert space if all functions $f$ with $\|f\|_{-1} = 0$ are identified with $0$. Denote by $\langle \cdot, \cdot \rangle_1$ and $\langle \cdot, \cdot \rangle_{-1}$ the inner products defined by polarizations in $\cH_{1}$ and $\cH_{-1}$, respectively. For more details about $\cH_1$ and $\cH_{-1}$, see \cite[\S2.2]{KLO12}. 

Pick some $\varphi \in E_0$ and consider a function $V^\varphi$ on $\dot E$ defined as 
$$V^\varphi(\dot v) \triangleq \frac{1}{2}\int_0^1 v(x)\varphi''(x)dx - \int_0^1 V'_x(v(x))\varphi(x)dx. $$
$V^\varphi$ is the drift of \eqref{spde} in the direction of $\varphi$, meaning that 
\begin{equation} \label{decomposition-clt}
\langle u(t), \varphi \rangle = \langle u(0), \varphi \rangle + \int_0^t V^\varphi(\dot u(r))dr + \langle W_t, \varphi \rangle. 
\end{equation}
Since both $\varphi''$ and $V'_x$ are bounded, $V^\varphi$ clearly belongs to $\cH$. Furthermore, for all $g \in \cE_A(\dot E)$ the integration-by-part formula for Wiener measure shows that 
$$\left|\langle V^\varphi, g \rangle_\pi\right| = \frac{1}{2}\left|\int_{\dot E} \langle -\varphi, Dg(\dot v) \rangle\pi(d\dot v)\right| \leq \frac{\sqrt{2}}{2}\|\varphi\|\|g\|_1, $$
which means that $V^\varphi \in \cH_{-1}$ and $\|V^\varphi\|_{-1} \leq \frac{\sqrt{2}}{2}\|\varphi\|$. 

Similar to \eqref{resolvent equation-core}, consider the resolvent equation for $\lambda > 0$: 
\begin{equation}\label{resolvent equation-clt}
\lambda f_\lambda^\varphi - \dot\cK f_\lambda^\varphi = V^\varphi. 
\end{equation}
Taking inner product with $f_\lambda^\varphi$ in both sides of \eqref{resolvent equation-clt} shows that $\|f_\lambda^\varphi\|_1^2 \leq \langle V^\varphi, f_\lambda^\varphi \rangle_\pi$. Since $\dot u(t)$ is reversible under $\pi$ and $V^\varphi \in \cH_{-1}$, we have 
\begin{equation} \label{estimate-clt}
\sup_{\lambda>0}\|\dot\cK f_\lambda^\varphi\|_{-1} = \sup_{\lambda>0}\|f_\lambda^\varphi\|_1 \leq \|V^\varphi\|_{-1} < \infty. 
\end{equation}
As the solution $f_\lambda^\varphi$ lies in $\dom(\dot\cK)$, we can consider the Dynkin's martingale 
$$M_\lambda^\varphi(t) \triangleq f_\lambda^\varphi(\dot u(t)) - f_\lambda^\varphi(\dot u(0)) - \int_0^t \dot\cK f_\lambda^\varphi(\dot u(r))dr. $$
Due to \cite[\S2.6]{KLO12}, under the condition \eqref{estimate-clt}, there exists some $f^\varphi \in \cH_1$ such that $f_\lambda^\varphi$ converges to $f^\varphi$ in $\cH_1$ as $\lambda \downarrow 0$. Now Corollary \ref{Ito} implies that 
\begin{equation} \label{Ito-clt}
\lim_{\lambda \downarrow 0} \mathbb{E}^\pi \left|M_\lambda^\varphi(t) - \int_0^t \langle Df^\varphi(\dot u(r)), dW_r \rangle\right|^2 = \lim_{\lambda \downarrow 0} 2t\|f_\lambda^\varphi - f^\varphi\|_1^2 = 0. 
\end{equation}
Furthermore, under the same condition, we have the error term $R_\lambda^\varphi(t) \triangleq \int_0^t V^\varphi(\dot u(r))dr - M_\lambda^\varphi(t)$ vanishes in the following sense (\cite[p. 51, Proposition 2.8]{KLO12}): 
\begin{equation} \label{vanishment of R}
\lim_{t \rightarrow \infty} \lim_{\lambda \rightarrow 0}\frac{1}{t}\mathbb{E}^\pi \left[R_\lambda^\varphi(t)^2\right] = 0. 
\end{equation}

We are now at the position to give the proof of Theorem \ref{clt}. Combining \eqref{decomposition-clt}, \eqref{Ito-clt}, and \eqref{vanishment of R} yields that as $t \rightarrow \infty$, the limit distribution of $\frac{1}{\sqrt{t}}\langle u(t), \varphi \rangle$ coincides with that of 
$$\frac{1}{\sqrt{t}}\bar{M}^\varphi(t) = \frac{1}{\sqrt{t}} \int_0^t \langle Df^\varphi(\dot u(r)) + \varphi, dW_r \rangle. $$
Now the ergodicity of $\pi$ together with the central limit theorem for continuous martingales (see, e.g., \cite{W07}) shows that for probability measure $\nu$ absolutely continuous with respect to $\mu$, 
\begin{equation} \label{clt-direction}
\lim_{t \rightarrow \infty}E_\nu \left|E_\bP \left[\left.f\left(\frac{\langle u(t), \varphi \rangle}{\sqrt{t}}\right)\right|\sF_0\right] - \int_{\bR} f(y)\Phi_{\sigma_\varphi}(dy)\right| = 0 
\end{equation}
for all $f \in C_b(\bR)$, where $\sigma_\varphi^2 = E_\pi \|Df^\varphi + \varphi\|^2$. 

Finally, observe that for $\varphi \in E_0$ such that $\int_0^1 \varphi(x)dx = 0$, we have $f^{\varphi} = -\langle v, \varphi \rangle$. Therefore, if we set $\varphi = e_j$ in \eqref{clt-direction}, $\sigma_{e_j}$ turns out to be $0$ for each $j \geq 1$. It implies that \eqref{clt-equation} holds with 
\begin{equation} \label{diffusion constant}
\sigma^2 = \sigma_{e_0}^2 = \lim_{\lambda \downarrow 0} E_\pi \|Df_\lambda^{e_0} + e_0\|^2. 
\end{equation}
In \eqref{diffusion constant}, $f_\lambda^{e_0} = (\lambda - \dot\cK )^{-1}V^{e_0}$, and its Fr\'echet derivative should be understood in the extended sense, as mentioned before Corollary \ref{Ito}. This completes the proof of Theorem \ref{clt}. 

Now fix $T > 0$ and consider the $E$-valued process $\{u^{(\epsilon)}(t) = \epsilon u(\epsilon^{-2}t)\}_{t \in [0, T]}$. To prove Theorem \ref{invariance principle}, it is sufficient to verify the tightness. 

\begin{prop} \label{tightness}
Suppose that the law of $u(0)$ is absolutely continuous with respect to $\mu$, then as $\epsilon \downarrow 0$, the laws of $u^{(\epsilon)}$ are tight in $C([0, T], E)$ with respect to the uniform topology. 
\end{prop}

\begin{proof}
Denote the three terms in the right-hand side of \eqref{mild solution} by $X(t)$, $Y(t)$, and $Z(t)$, respectively. Furthermore, let 
$$X^\perp(t) \triangleq X(t) - \int_0^1 X(t, x)dx = \sum_{j=1}^\infty \langle X(t), e_j \rangle \cdot e_j. $$
We define $Y^\perp(t)$ and $Z^\perp$ similarly, then due to \eqref{mild solution}, 
\begin{equation} \label{decomposition-tightness}
u^{(\epsilon)}(t) = \epsilon\langle u(\epsilon^{-2}t), e_0 \rangle + \epsilon X^\perp(\epsilon^{-2}t) + \epsilon Y^\perp(\epsilon^{-2}t) + \epsilon Z^\perp(\epsilon^{-2}t). 
\end{equation}
For the first term in the right-hand side of \eqref{decomposition-tightness}, \eqref{decomposition-clt} yields that 
$$\epsilon\langle u(\epsilon^{-2}t), e_0 \rangle = \epsilon\langle u(0), e_0 \rangle - \epsilon\int_0^{\epsilon^{-2}t}V^{e_0}(\dot u(r))dr + \epsilon\langle W_{\epsilon^{-2}t}, e_0 \rangle. $$
By \cite[p. 74, Theorem 2.32]{KLO12}, this term is tight when $\epsilon \downarrow 0$. Meanwhile, as the heat semigroup is contractive, $|\epsilon X^\perp(\epsilon^{-2}t, \cdot)|_\infty \leq \epsilon|u|_\infty$, therefore the second term in the right-hand side of \eqref{decomposition-tightness} vanishes uniformly when $\epsilon \downarrow 0$. 

The tightness of the last two terms can be verified by the following estimates. For all $p >1$, there exists a finite constant $C = C(p) > 0$ such that for all $t_1$, $t_2 \in [0, \infty)$, $x_1, x_2 \in [0, 1]$, 
\begin{equation} \label{estimate-tightness-1}
E\left|Y^\perp(t_1, x_1) - Y^\perp(t_2, x_2)\right|^{2p} \leq C(|t_1 - t_2|^p + |x_1 - x_2|^p); 
\end{equation}
\begin{equation} \label{estimate-tightness-2}
E\left|Z^\perp(t_1, x_1) - Z^\perp(t_2, x_2)\right|^{2p} \leq C(|t_1 - t_2|^{\frac{p}{2}} + |x_1 - x_2|^p). 
\end{equation}
We only prove \eqref{estimate-tightness-1} here. For \eqref{estimate-tightness-2}, since $Z^\perp$ are Gaussian variables, we only need to prove for $p = 1$ and the argument is similar. Noticing that $S(t)e_j = e^{-\pi^2j^2t}e_j$ for $j \geq 0$, we first suppose that $t_1 = t_2 = t$ and write $\delta_x = |x_1 - x_2|$, then 
$$E\left|Y^\perp(t, x_1) - Y^\perp(t, x_2)\right|^{2p} \leq C_1\left[\sum_{j=1}^\infty\int_0^t \left|e^{-\pi^2j^2(t - r)}\sin(\pi j \delta_x)\right|dr\right]^{2p}. $$
Let $\alpha = \lfloor\delta_x^{-\frac{1}{2}}\rfloor$, the right-hand side is bounded by 
$$C_1\left[\sum_{j=1}^\alpha \frac{\delta_x}{\pi j} + \sum_{j=\alpha+1}^\infty \frac{1}{\pi^2j^2}\right]^{2p} \leq C_1\left[\frac{\alpha\delta_x}{\pi} + \frac{1}{\pi^2\alpha}\right]^{2p} \leq C_2\delta_x^p. $$
Next let $x_1 = x_2 = x$ in \eqref{estimate-tightness-1}, write $\delta_t = t_1 - t_2$ and compute similarly, 
$$E\left|Y^\perp(t_1, x) - Y^\perp(t_2, x)\right|^{2p} \leq C_3\left[\sum_{j=1}^\infty \frac{1 - e^{-\pi^2j^2\delta_t}}{\pi^2j^2}\right]^{2p} \leq C_4\delta_t^p. $$
Therefore \eqref{estimate-tightness-1} holds for all $t_1$, $t_2 \in [0, \infty)$ and $x_1$, $x_2 \in [0, 1]$. 

Due to \eqref{estimate-tightness-1} and \eqref{estimate-tightness-2}, the last two terms in \eqref{decomposition-tightness} are tight in $C([0, T]\times[0, 1], \bR)$. Since the topology of $C([0, T]\times[0, 1], \bR)$ and $C([0, T], E)$ are equivalent, they are tight in $C([0, T], E)$. 
\end{proof}

Proposition \ref{tightness} together with Theorem \ref{clt} completes the proof of Theorem \ref{invariance principle}. 

\end{section}

\section*{Acknowledgements}
The author greatly thanks Professor Tadahisa Funaki for his instructive advice and suggestions. The author also greatly thanks Professor Da Prato who kindly suggested him the arguments in \cite{PT01}, which simplified the original proofs. The author is supported by Leading Graduate Course for Frontiers in Mathematical Sciences and Physics (FMSP), MEXT, Japan. 

\bibliography{[BIB]clt_for_spde.bib}

\end{document}